\documentclass{amsart}

\usepackage{amssymb,latexsym}
\usepackage[utf8]{inputenc}
\usepackage[english]{babel}
\usepackage{amsmath,amsthm}
\usepackage{enumerate}
\usepackage{enumitem}
\usepackage{listings}
\usepackage{url}
\newcommand{\subscript}[2]{$#1 _ #2$}
\lstdefinestyle{customc}{
	belowcaptionskip=1\baselineskip,
	breaklines=true,
	frame=L,
	xleftmargin=\parindent,
	language=C,
	showstringspaces=false,
	basicstyle=\footnotesize\ttfamily,
	keywordstyle=\bfseries\color{green!40!black},
	commentstyle=\itshape\color{purple!40!black},
	identifierstyle=\color{blue},
	stringstyle=\color{orange},
}

\lstdefinestyle{customasm}{
	belowcaptionskip=1\baselineskip,
	frame=L,
	xleftmargin=\parindent,
	language=[x86masm]Assembler,
	basicstyle=\footnotesize\ttfamily,
	commentstyle=\itshape\color{purple!40!black},
}

\newcommand{\myspace}{1.5cm}

\theoremstyle{plain}
\newtheorem{theorem}{Theorem}
\newtheorem{corollary}{Corollary}

\newtheorem{lemma}{Lemma}

\theoremstyle{definition}
\newtheorem{definition}{Definition}

\theoremstyle{remark}

\numberwithin{equation}{section}

\DeclareMathOperator{\AGM}{agm}
\DeclareMathOperator{\AM}{AM}
\DeclareMathOperator{\GM}{GM}

\begin{document} %

\title{Means Compatible with Semigroup Laws}

\author{R. Padmanabhan}
\address{
   Department of Mathematics\\
   University of Manitoba\\
   Winnipeg, Manitoba  R3T 2N2\\
   Canada}
\email{padman@cc.umanitoba.ca}
\urladdr{http://home.cc.umanitoba.ca/\~{}padman}

\author{Alok Shukla}
\address{
    Department of Mathematics\\
    University of Manitoba\\
    Winnipeg, Manitoba  R3T 2N2\\
    Canada}
\email{Alok.Shukla@umanitoba.ca}
\urladdr{https://www.math.umanitoba.ca/people/pages/alok-shukla/}



\date{}

\begin{abstract}
A binary mean operation $ m(x,y) $ is said to be compatible with a semigroup law $ * $,  if $ * $ satisfies the Gauss' functional equation $ m(x,y) * m(x,y) = x * y $ for all $ x, y $. Thus the arithmetic mean is compatible with the group addition in the set of real numbers, while the geometric mean is compatible with the group multiplication in the set of all positive real numbers. Using one of Jacobi's theta functions, Tanimoto \cite{tanimoto2007noveljp}, \cite{tanimoto2007novel} has constructed a novel binary operation $ * $ corresponding to the arithmetic-geometric mean $ \AGM (x,y) $ of Gauss. Tanimoto shows that it is only a loop operation, but not associative. A natural question is to ask if there exist a group law $ * $  compatible with arithmetic-geometric mean. In this paper we prove that there is no semigroup law compatible with $ \AGM $ and hence, in particular, no group law either. Among other things, this explains why Tanimoto's novel operation $ * $ using theta functions must be non-associative.
\end{abstract}

\maketitle

\section{Introduction} \label{sec-intro}

Gauss discovered the arithmetico-geometric mean ($ \AGM $) at the age of $ 15 $. Starting with two positive real numbers $ x $ and $ y $, Gauss considered the sequences $ \{x_n\} $ and $ \{y_n\} $  of arithmetic and geometric means   $$ x_0 =x,  y_0 = y , x_n = \dfrac{x_{n-1}+y_{n-1}}{2},  y_n = \sqrt{x_{n-1} y_{n-1}},  {\text{ for }}   n \geq 1 .$$
Then Gauss defined $ \AGM (x,y) $ to be the common limit of the sequences $ \{x_n\} $ and $ \{y_n\} $, i.e.,
\begin{align}
\AGM (x,y)  = \lim\limits_{n \to \infty} x_n = \lim\limits_{n \to \infty} y_n.
\end{align}
For an engaging historical account on $ \AGM $ and its applications in mathematics readers are referred to  \cite{almkvist1988gauss},\cite{cox1997arithmetic}.

In this paper, we ask if there exist a group law $ * $, which is compatible with $ \AGM $. Before proceeding further we give some definitions relevant to this work.
\begin{definition}[\textbf{Mean}] \label{def:mean}
	Let $ S $ be a set equipped with a binary operation $ m $. It is said that $ m $ is a \textit{mean}, if it satisfies the following 
	\begin{enumerate} [label=(\subscript{M}{{\arabic*}})]
		\item \hspace{\myspace}       $ 	m(x,x) = x $,       
		\item \hspace{\myspace} $ 	m(x,y) = m(y,x) $, 
		\item \hspace{\myspace} $ 	m(x,y) = m(z,y) \implies x = z. $ 
	\end{enumerate}
\end{definition}

\begin{definition}[\textbf{Compatibility of binary operations}] \label{def:compatible}
	Let $ S $ be a set equipped with a binary mean operation $ m $ and another binary operation $ * $. 
	The binary mean operation $ m $, and the binary operation $ * $, are said to be \textit{compatible} with each other,  if $ m(x,y) * m(x,y) = x * y $ for all $ x, y \in S $.
\end{definition}

Here we find conditions on the mean $ m $ which force any compatible operation $ * $ to be a group operation.

Let $ \displaystyle \AM(x,y) = \frac{x+y}{2} $ be the arithmetic mean of $ x,y \in \mathbb{R} $ with $ + $ being the usual addition in $\mathbb{R} $. Then clearly $ \AM(x,y)+ \AM(x,y)=x+y $, therefore, the classical arithmetic mean $ \AM(x,y)$ is compatible with the group law of $ + $ in $\mathbb{R} $, in the sense of Def.~$\ref{def:compatible}  $.  Similarly, the geometric mean $ \GM $ is also compatible with the group law of multiplication in positive reals. Similarly, it can be verified that the harmonic mean $ \displaystyle h(x,y) = \frac{2xy}{x+y} $ is compatible with the group law $\displaystyle x*y = \frac{xy}{x+y} $. It is then natural to consider if there exists any such group operation over $ \mathbb{R}^{+} $, which is compatible with the arithmetic-geometric mean ($ \AGM $) of Gauss. In other words, we want to address the question, if there exists a group operation $ * $, such that $ \AGM(x,y)* \AGM(x,y)=x*y $.  Using one of Jacobi's theta functions, Shinji Tanimoto has successfully constructed a non-associative loop operation $ \star $ (c.f. \cite{tanimoto2007noveljp}, \cite{tanimoto2007novel}, Sec.~$ \ref{subsec:Tanimoto} $ below) that is compatible with $ \AGM $.  However, no group law $ * $  compatible with $ \AGM $ is known to exist. Indeed, we prove that no such group law $ * $ can exist, which is compatible with $\AGM $.

\subsection{A non-associative loop operation compatible with $ \AGM $}\label{subsec:Tanimoto}
Now we recall the binary operation $ \star $ introduced by  Shinji Tanimoto in \cite{tanimoto2007noveljp}, \cite{tanimoto2007novel}. 

\begin{definition}[Tanimoto, \cite{tanimoto2007noveljp}, \cite{tanimoto2007novel}] \label{def:starTanimoto}
	{\it 
		For any two positive numbers $x$ and $y,$ choose a unique $q~(-1 <q<1)$ such that 
		$1 / {\rm agm}(x,y) = \theta^2(q)$. Here,
		 $\theta$ is one of the Jacobi's theta functions:
		 \[
		 \theta(q) = \sum_{n = - \infty}^{+ \infty} q^{n^2} = 1 + 2\sum_{n = 1}^{\infty} q^{n^2}.
		 \]
		Then define
		\begin{eqnarray}
		x \star y =  \theta^2(-q)/\theta^2(q).
		\end{eqnarray} 
	}
\end{definition}
We also recall the following theorems from \cite{tanimoto2007novel}, which describe the properties of the $ \star $ operation. We note that here variables $x, y$ are positive real numbers.
\noindent
\begin{theorem}[Tanimoto, \cite{tanimoto2007novel}]\label{thm:Tanimoto1}
 {\it The operation $\star$ defined above satisfies the following properties.}\\
	~{\bf (A)} $1 \star x = x$ {\rm for all} $x$.  {\it Hence {\rm 1} is the unit element of the operation}. \\
	~{\bf (B)} $x \star x = y \star y$ {\it implies} $x = y$. \\
	~{\bf (C)} $x \star y = {\rm agm} (x, y) \star {\rm agm} (x, y)$. {\it Thus the mean with respect to
		the operation is the $ \AGM $.} 
\end{theorem}
\begin{theorem}[Tanimoto, \cite{tanimoto2007novel}] \label{thm:Tanimoto2}
	  {\it The operation $\star$ satisfies the following algebraic properties.}\\
	~{\bf (D)}  $a \star x = a \star y$ {\it implies} $x = y$ ({\it a cancellation law}). \\
	~{\bf (E)}  $(ax) \star (ay) = a \star (a(x \star y))$ {\it for any} $a, x, y$ ({\it a distributive law}).\\
	~{\bf (F)}  {\it If $z = x \star y$, then} $y  = x(x^{-1} \star (x^{-1}z))$. {\it In particular, the inverse of $x$ 
		with respect to the operation is} $x(x^{-1} \star x^{-1})$.
\end{theorem}

Finally, we note that Tanimoto claims that the $\star  $ operation is not associative (although, he does not give any example).\footnote{It can easily be verified that $ \AGM(\AGM(1,2), \AGM(3,4)) \neq \AGM(\AGM(1,3), \AGM(2,4)) $. Theorem \ref{thm_main}, then implies that $ \star $ is not associative.}

\section{Main results}
Now we are ready to prove our claim that there does not exist any group law $ * $, that is compatible with $ \AGM $ in the sense of the Def.~$ \ref{def:compatible} $.
In this direction, first we prove the following theorem.
\begin{theorem} \label{thm_main}
	Let $ m(x,y) $ be a binary operation defined over positive reals satisfying the following:
	\begin{enumerate} [label=(\subscript{M}{{\arabic*}})]
		\item \hspace{\myspace}       $ 	m(x,x) = x $,       \label{cond_one}
		\item \hspace{\myspace} $ 	m(x,y) = m(y,x) $,  \label{cond_two}
		\item \hspace{\myspace} $ 	m(x,y) = m(z,y) \implies x = z $, \label{cond_three}
		\item \hspace{\myspace} $ 	m(x,y) * m(x,y) = x * y $     \qquad (Gauss' Functional Equation),\label{cond_four}
		\item \hspace{\myspace} $ 	e*x = x $, \label{lem:idempotent} 
		\item \hspace{\myspace} $ x * x = y * y  \implies x = y. $ \label{cond_six}
	\end{enumerate}
	Then $ m $ is medial, i.e., $ m(m(x,y),m(z,u)) = m(m(x,z),m(y,u)) $ if and only if the $ * $ operation is associative.
\end{theorem}

Before proving Theorem \ref{thm_main}, we state and prove the following lemmas.
\begin{lemma}
	Under the hypothesis $ (M_1) $-$ (M_6) $ of Theorem~\ref{thm_main}, we have the following results.
	\begin{enumerate}[]
			\item $ m(x,y) = m(e,x*y)$, \label{lem:tostar}
			\item $ x*y = y*x $. \label{lem:commute}
	\end{enumerate}
\end{lemma}
\begin{proof}
	The lemma follows from the following calculations.
	\begin{enumerate}[]
			\item We have
			\begin{align*}
			m(x,y)*m(x,y) &= x*y \qquad (\text{from  \ref{cond_four}})  \\
			&= e*(x*y) \qquad (\text{from  \ref{lem:idempotent}})  \\
			&= m(e,x*y)*m(e,x*y) \qquad (\text{from  \ref{cond_four}}) . 
			\end{align*}
			Now the result follows from  \ref{cond_six}.
			\item $ x*y = m(x,y)*m(x,y) = m(y,x)*m(y,x) = y*x .$
	\end{enumerate}
	
\end{proof}

\begin{lemma}\label{lem:cond_five}
	Assume the hypothesis $ (M_1) $-$ (M_6) $ of Theorem~\ref{thm_main}. Also assume either  $ * $ is associative, or $ m $ is medial. Then
	\begin{align} 
	   m(x,e) * m(e,y) = m(x,y). 
	\end{align}

\end{lemma}
\begin{proof}
	First we assume that $ * $ is associative. Then the desired conclusion follows from the following calculation and   \ref{cond_six}.
		\begin{align*}
		&(m(x,e) * m(e,y))*(m(x,e) * m(e,y)) \\
		&= m(x,e) * m(e,y)*m(x,e) * m(e,y)  \qquad (\text{from the associativity of $ * $})  \\
		&= m(x,e) * m(x,e)*m(e,y) * m(e,y)  \qquad (\text{from Lemma 1, part \ref{lem:commute}})  \\
		&= (m(x,e) * m(x,e))*(m(e,y) * m(e,y))  \qquad (\text{from the associativity of $ * $})  \\
		&=  (x*e)*(e*y) \qquad (\text{from  \ref{cond_four}})  \\
		&= x*y  \qquad (\text{from Lemma 1, part \ref{lem:commute} and  \ref{cond_four}})  \\
		& = m(x,y) *m(x,y) \qquad (\text{from  \ref{cond_four}})   
		\end{align*}  
Next we assume that $ m $ is medial, i.e., $ m(m(x,y),m(z,u)) = m(m(x,z),m(y,u)) $.
Then we have 
\begin{align}
	& \, \,m(m(x,y),m(z,u))* m(m(x,y),m(z,u)) \nonumber\\ 
	&= m(m(x,z),m(y,u)) * m(m(x,z),m(y,u)) \nonumber \\
	\implies & m(x,y) * m(z,u) = m(x,z)*m(y,u) \qquad (\text{from  \ref{cond_four}}) \label{eq_one}  \\ 
	\implies & m(x,y) * m(e,e) = m(x,e)*m(y,e)  \qquad (\text{put $ z=u=e $}) \nonumber \\
	\implies & m(x,y) * e = m(x,e)*m(e,y) \qquad (\text{from  \ref{cond_one} and \ref{cond_two}})  \nonumber \\
	\implies & m(x,y)  = m(x,e)*m(e,y) \qquad (\text{from  \ref{cond_two} and Lemma 1, \ref{lem:commute}})  \nonumber 
\end{align}

\end{proof}

\paragraph{\textbf{Proof of Theorem~\ref{thm_main}}}
\begin{proof}
	Assume that $ * $ is associative. Then
	\begin{align}
	& \, \,m(m(x,y),m(z,u))* m(m(x,y),m(z,u)) \nonumber \\
	&= m(x,y) * m(z,u) \qquad (\text{from  \ref{cond_four}})  \nonumber  \\
	& = (m(x,e)*m(e,y))*(m(z,e)*m(e,u)) \qquad (\text{from Lemma \ref{lem:cond_five}}) \nonumber \\
	&= m(x,e)*m(e,y)*m(z,e)*m(e,u) \qquad (\text{from the associativity of $ * $})  \nonumber \\
	&= m(x,e)*m(z,e)*m(e,y)*m(e,u) \qquad (\text{from Lemma 1, part \ref{lem:commute}})  \nonumber\\
	&= m(x,e)*m(e,z)*m(y,e)*m(e,u) \qquad (\text{from  \ref{cond_two}}) \nonumber \\
	&= (m(x,e)*m(e,z))*(m(y,e)*m(e,u)) \qquad (\text{from the associativity of $ * $}) \nonumber\\
	&= m(x,z)*m(y,u) \qquad (\text{from Lemma \ref{lem:cond_five}}) \nonumber \\
	&= m(m(x,z),m(y,u))* m(m(x,z),m(y,u)) \qquad (\text{from  \ref{cond_four}})  \nonumber. 
	\end{align}
	This proves one direction of the theorem, as  \ref{cond_six} now implies that $ m $ is medial, i.e., $ m(m(x,y),m(z,u)) = m(m(x,z),m(y,u)) $.\\
	
	Next to prove the other direction assume that $$ m(m(x,y),m(z,u)) = m(m(x,z),m(y,u)) .$$
	Then from Eq.~$ \eqref{eq_one} $ we have
	\begin{align*}
	m(x,y) * m(z,u) = m(x,u)*m(z,y).
	\end{align*}
	For $ x=e $, the above relation becomes
	\begin{align} \label{eq:basic1}
	m(e,y) * m(z,u) = m(e,u)*m(z,y).
	\end{align}
	Now, 
	\begin{align}\label{eq:basic2}
	m(e,y) * m(z,u)   =& m(e,y) *m(e,z*u) \qquad (\text{from Lemma 1, part \ref{lem:tostar} }) \nonumber\\
		=& m(y,e) *m(e,z*u) \qquad (\text{from  \ref{cond_two}})   \nonumber\\
	=& m(y,z*u)          \qquad (\text{from Lemma \ref{lem:cond_five}}) \nonumber \\
	=& m(e,y*(z*u)). \qquad (\text{from Lemma 1, part \ref{lem:tostar} }) 
	\end{align}
	Similarly,
	\begin{align} \label{eq:basic3}
	m(e,u) * m(z,y)   = m(e,u*(z*y)).
	\end{align}
	
	From Eq.~$ \eqref{eq:basic1}$, Eq.~$ \eqref{eq:basic2}$, and Eq.~$ \eqref{eq:basic3} $, we get $$m(e,y*(z*u)) =  m(e,u*(z*y)) .$$ 
	\begin{align*}
	m(e,y*(z*u)) &=  m(e,u*(z*y)) \\
	\implies y*(z*u) &= u*(z*y) \qquad (\text{from  \ref{cond_three}})   \\
 \implies    y*(z*u)	&=  u*(y*z)        \qquad (\text{from Lemma 1, part \ref{lem:commute} }) \\
  \implies    y*(z*u)	&=  (y*z)*u      \qquad (\text{from Lemma 1, part \ref{lem:commute} }) 
	\end{align*}

	This completes the proof.
	
\end{proof}

\begin{corollary} [of Theorem \ref{thm_main}]
There does not exist any group law $ * $, that is compatible with $ \AGM $.
\end{corollary}
\begin{proof}
From the definition of $ \AGM $, it is obvious that $ \AGM(x,x) =x $ and $ \AGM(x,y) = \AGM(y,x) $. Further, if $\AGM(x,y) = \AGM(x,z)  $, then $$ \AGM(x,y) \star \AGM(x,y) = \AGM(x,z) \star \AGM(x,z) \implies x \star y = x \star z  \implies y = z,$$ from Theorem~$ \ref{thm:Tanimoto1} $\,(C) and Theorem~$ \ref{thm:Tanimoto2} $\,(D).
Therefore, $ \AGM $ is a  mean operation in the sense of Def.~$ \ref{def:mean} $.
Further, the $ \star $ operation defined by Tanimoto (see Def.~$ \ref{def:starTanimoto} $) is not associative, and moreover, $ \AGM $ and $ \star $ satisfy the hypothesis of Theorem~$ \ref{thm_main} $, from the definition of $ \star $, and by virtues of Theorem~$ \ref{thm:Tanimoto1} $ and Theorem~$ \ref{thm:Tanimoto2} $ (we note that in \cite{tanimoto2007novel}, our identity element $ e  $ is represented by $ 1 $). Therefore, from Theorem~$ \ref{thm_main} $, it follows that $ \AGM $ is not medial (alternatively, from a direct numerical computation it can be verified that $ \AGM $ is not medial).  But then, Theorem~$ \ref{thm_main} $ also implies that $ \AGM $ can not be compatible with any $ * $ operation which is associative and satisfies \ref{cond_four}-\ref{cond_six}. Therefore, there can not exist any group law $ * $, that is compatible with $ \AGM $. 
\end{proof}

Suppose for a mean $ m $, if $ m(m(x,y),m(x,z)) = m(x,m(x,z)) $, then the mean $ m $ is said to be self-distributive. If $ (x*x)*(y*z) = x*y*(x*z) $ then $ * $ is  called Moufang.

It is easy to see that in the above proofs, the full force of associativity (or, for that matter the medial law)  is not used. Indeed, `associativity' and `medial' in  Theorem~\ref{thm_main}, can be replaced by `Moufang' and `self-distributive', respectively and the proof of the theorem still remains valid.

\begin{theorem} \label{thm:Moufang}
	For a mean $ m $ and satisfying $ (M_1) $-$ (M_6) $ of Theorem~\ref{thm_main},
	$ m $ is self-distributive, i.e.,  $ m(m(x,y),m(x,z)) = m(x,m(x,z)) $ if and only if the $ * $ operation is Moufang.
\end{theorem}
One can easily verify (for example by using Mathematica) that 
$$ \AGM(\AGM(1,2),\AGM(1,3)) \neq \AGM(1,\AGM(2,3)) .$$
Hence, Gauss' Functional Equation for $ \AGM $ can not be solved even among Moufang loops.

Although, we have remarked earlier that the proof of Theorem~\ref{thm:Moufang} follows on the same line as Theorem~\ref{thm_main}, we are enclosing an automated proof of this theorem by using Prover9 \cite{prover9}, in the Appendix, for readers interested in automated reasoning.\\

\section{Appendix}
\noindent\textbf{Moufang identity implies self-distrtibutivity.}
\begin{verbatim}
1 m(x,m(y,z)) = m(m(x,y),m(x,z)) # label(goal).  [].
3 m(x,y) = m(y,x).  [].
5 m(x,y) * m(x,y) = x * y.  [].
6 x * x != y * y | x = y.  [].
7 x * e = x.  [].
8 (x * y) * (x * z) = (x * x) * (y * z).  [].
9 m(m(c1,c2),m(c1,c3)) != m(c1,m(c2,c3)).  [1].
10 m(c1,m(c2,c3)) != m(m(c1,c2),m(c1,c3)).  [9].
15 m(x,y) * m(y,x) = y * x.  [3,5].
16 x * y = y * x.  [3,5,15].
17 x * y != z * z | m(x,y) = z.  [5,6].
23 c1 * m(c2,c3) != m(c1,c2) * m(c1,c3).  [6,10,5,5].
32 c1 * m(c3,c2) != m(c1,c2) * m(c1,c3).  [3,23].
48 e * x = x.  [16,7].
50 (x * y) * (z * x) = (x * x) * (y * z).  [16,8].
79 c1 * m(c3,c2) != m(c2,c1) * m(c1,c3).  [3,32].
130 m(e,x * x) = x.  [17,48].
132 m(x * x,y * y) = x * y.  [17,8].
160 m(e,x * y) = m(x,y).  [5,130].
221 c1 * m(c3,c2) != m(c1,c3) * m(c2,c1).  [16,79].
293 m(x * y,z * z) = m(x,y) * z.  [5,132].
294 m(x * x,y * z) = x * m(y,z).  [5,132].
662 m(x * y,z * x) = x * m(y,z).  [50,160,160,294].
1706 m(x * y,z * u) = m(x,y) * m(z,u).  [5,293].
1748 m(x,y) * m(z,x) = x * m(y,z).  [662,1706].
1749 $F.  [1748,221].
\end{verbatim}

\noindent\textbf{Self-distributivity implies Moufang identity.}
\begin{verbatim}
1 (x * y) * (x * z) = (x * x) * (y * z) # label(non_clause) # label(goal).  [].
2 m(x,x) = x.  [].
3 m(x,y) = m(y,x).  [].
4 m(x,y) != m(z,y) | x = z.  [].
5 m(x,y) * m(x,y) = x * y.  [].
6 x * x != y * y | x = y.  [].
7 x * e = x.  [].
8 m(x,m(y,z)) = m(m(x,y),m(x,z)).  [].
9 m(m(x,y),m(x,z)) = m(x,m(y,z)).  [8].
10 (c1 * c2) * (c1 * c3) != (c1 * c1) * (c2 * c3).  [1].
13 m(x,y) != m(z,x) | y = z.  [3,4].
15 m(x,y) * m(y,x) = y * x.  [3,5].
16 x * y = y * x.  [3,5,15].
17 x * y != z * z | m(x,y) = z.  [5,6].
22 m(x,y) * m(x,z) = x * m(y,z).  [9,5,9,5].
24 e * x = x.  [16,7].
26 m(x,y) != m(x,z) | y = z.  [3,13].
29 m(e,x * x) = x.  [17,24].
33 x * x != y | m(e,y) = x.  [24,17].
41 m(e,x) != y | y * y = x.  [29,26].
55 x != y | y * y = x * x.  [29,41].
56 m(e,x * y) = m(x,y).  [33,22,2].
58 m(x * x,y) = x * m(e,y).  [29,22,24].
68 m(x,e) != m(y,z) | y * z = x.  [56,13].
74 m(x,y) * m(e,z) = m(x * y,z).  [56,22,24].
79 m(x,y * y) = y * m(e,x).  [58,3].
80 m(x * x,y) = x * m(y,e).  [3,58].
99 m(x * x,y * z) = x * m(y,z).  [56,58].
134 m(x,y * y) = y * m(x,e).  [3,79].
153 m(x,x * y) = x * m(y,e).  [80,22,22,2,3].
220 m(x * x,y) = m(x,x * y).  [153,80].
225 m(x,y * y) = m(y,y * x).  [153,134].
240 m(x,x * (y * z)) = x * m(y,z).  [99,220].
338 x * (y * y) = y * (y * x).  [55,225,22,2,22,2].
427 (x * x) * y = x * (x * y).  [338,16].
448 (c1 * c2) * (c1 * c3) != c1 * (c1 * (c2 * c3)).  [10,427].
504 m(c1 * c2,c1 * c3) != c1 * m(c2,c3).  [68,448,3,56,240].
550 m(x,y) * m(z,u) = m(x * y,z * u).  [56,74].
568 m(x * y,x * z) = x * m(y,z).  [22,550].
569 $F.  [568,504].
\end{verbatim}

\bibliographystyle{plain}

%


\end{document}